\documentclass[10pt]{amsart}

\usepackage{amscd}
\usepackage{amssymb}
\usepackage{latexsym}
\usepackage{graphicx}
\usepackage[parfill]{parskip} 
\usepackage{amsmath}
\usepackage{amsfonts}
\usepackage{amssymb}
\usepackage{latexsym}
\usepackage{graphicx}
\usepackage{enumerate}
\usepackage{bbm}
\usepackage{fourier}

\usepackage[utf8]{inputenc}

\usepackage[final]{hyperref}
\hypersetup{colorlinks=true, linkcolor=blue, anchorcolor=blue, citecolor=red, 
filecolor=blue, menucolor=blue, pagecolor=blue, urlcolor=blue}

\newcommand{\norm}[1]{\left\Vert #1 \right\Vert}

\newcommand{\ls}{\lesssim}

\newcommand{\C}{\mathbb{C}}

\newcommand{\R}{\mathbb{R}}

\newcommand{\angles}[1]{\langle #1 \rangle}

\newtheorem*{PWtheorem}{Paley-Wiener Theorem}

\DeclareMathOperator{\im}{Im}
\DeclareMathOperator{\re}{Re}

\newtheorem{theorem}{Theorem}

\newtheorem{lemma}{Lemma}

\theoremstyle{definition}
\newtheorem{definition}{Definition}

\theoremstyle{remark}

\numberwithin{equation}{section}
\begin{document}

\title[Cubic NLS: spatial analyticity]{ On the radius of spatial analyticity for cubic  nonlinear Schr\"{o}dinger equation }

\author{Achenef Tesfahun}
\email{Achenef.Temesgen@uib.no}
 
\address{Department of Mathematics\\
University of Bergen\\
PO Box 7803\\
5020 Bergen\\ Norway}




\date{July 26, 2016}

\begin{abstract}
  It is shown that the uniform radius of spatial analyticity $\sigma(t)$ of solutions at time $t$ to the 1d, 2d and 3d cubic nonlinear Schr\"{o}dinger equations cannot decay faster than 
$1/|t|$ as $|t| \to \infty$, given initial data that is analytic with fixed radius $\sigma_0$.
\end{abstract}

\maketitle

\section{Introduction}
We consider the Cauchy problem for the defocusing cubic nonlinear Schr\"{o}dinger equation (NLS)
\begin{equation}\label{nls}
\begin{cases}
iu_t + \Delta u = |u|^2 u  ,
\\
u(x,0) = u_0(x),
\end{cases}
\end{equation}
where $u: \R^{1+d}\rightarrow \C$.
A solution to \eqref{nls} satisfies 
 \begin{align}
 \label{consv-mass}
 M[u[t]]:&=\norm{u(t)}^2_{L^2(\R^d)}=M[u[0]] 
 \end{align}
 and
 \begin{align}
 \label{consv-energy}
 E[u(t)]:&=\| \nabla u (t) \|^2_{L^2 (\R^d)}+ \frac12\|  u(t)\|_{L^4(\R^d)}^4=E[u(0)]  
 \end{align}
 which are the conservation of mass and energy, respectively. The well-posedness of \eqref{nls} in Sobolev spaces $H^s(\R^d)$ has been studied 
intensively; see for instance \cite{CKST02, CW90, FG07, Ts87}.  In particular, global well-posedness is known
for $s\ge (d-1)/2$ for $d=1, 2,3$.

 In the present paper we shall study spatial analyticity of the solutions to \eqref{nls} motivated by earlier works on this issue 
 for the derivative NLS in 1d by Bona, Gruji\'c and Kalisch \cite{bgk06}. 
 In particular, we consider a real-analytic initial data $u_0$ with 
 uniform radius of analyticity $\sigma_0>0$, so there is a holomorphic extension to a complex strip 
 $$ S_{\sigma_0} =\{x + iy : |y| < \sigma_0 \}.$$ The question is then whether this property persists for all later times $t$, but with a possibly smaller and shrinking radius of analyticity $\sigma(t) > 0$, i.e. is the solution $u(t,x)$ of \eqref{nls} analytic in $S_{\sigma(t)}$ for all $t$? 
   For short times it is shown that the radius of analyticity remains at least as large as the initial radius, i.e. one can take $\sigma(t)=\sigma_0$. 
For large times on the other hand we use the idea introduced in \cite{st15} (see also \cite{sd16}) to show that $\sigma(t)$ can decay no faster than 
$1/|t|$ as $|t| \to \infty$.
 For studies on related issues for nonlinear partial differential equations see for instance \cite{bgk05, cdn14,  hhp11,  hp12,  j86, km86, p12}.

A class of analytic function spaces suitable to study analyticity of solution is the analytic Gevrey class (see e.g. \cite{ft89}). These spaces are denoted $G^{\sigma,s} = G^{\sigma,s}(\mathbb R^d)$ with a norm given by
\[
\| f \|_{G^{\sigma, s}} = \left\| e^{\sigma | D |}\angles{ D }^s f\right\|_{L^2},
\]
where $D = -i\nabla$ with Fourier symbol $\xi$ 
and $\angles{\cdot}=\sqrt{1+|\cdot|^2}$.
 space, denoted $G^{\sigma,s} = G^{\sigma,s}(\mathbb R^d)$, is a . This space

 For $\sigma=0$ the Gevrey-space coincides with the Sobolev space $H^s$.  
One of the key properties of the Gevrey space is that every function in $G^{\sigma,s}$ with $\sigma>0$ has an analytic extension to the strip $S_\sigma$. This property 
is contained in the following.
\begin{PWtheorem}
Let $\sigma > 0$, $s \in \mathbb R$.  Then the following are equivalent:
\begin{enumerate}
\item $f \in G^{\sigma, s}$.
\item $f$ is the restriction to the real line of a function $F$ which is holomorphic in the strip
\[
S_\sigma =  \{ x + i y :\ x,y \in \mathbb{R}^d, \  | y | < \sigma\}
\]
and satisfies
\[
\sup_{| y | < \sigma} \| F( x + i y ) \|_{H^s_x} < \infty.
\]
\end{enumerate}
\end{PWtheorem}
The proof given for $s = 0$ in \cite[p. 209]{k76} applies also for $s\in \R$ with some obvious
modifications.

Observe that the Gevrey spaces satisfy the following embedding property:
\begin{align}
\label{Gembedding}
  G^{\sigma,s} &\subset G^{\sigma',s'} \quad \text{for all $0 \le \sigma' < \sigma$ and $s,s' \in \mathbb R$}.
\end{align}
In particular, setting $\sigma'=0$, we have the embedding $ G^{\sigma,s} \subset H^{s'}$ for all $0  < \sigma$ and $s, s'\in \mathbb R $.
As a consequence of this property and the existing well-posedness theory in $H^s$ we conclude that the Cauchy problem \eqref{nls} has a unique, 
smooth solution for all time, given initial data 
$u_0 \in G^{\sigma_0,s}$ for all $\sigma_0>0$ and $s\in \R$.

Our main result gives an algebraic lower bound on the radius of analyticity $\sigma(t)$ of the solution as the time $t$ tends to infinity.
\begin{theorem}\label{thm-gwp}
Assume $u_0 \in G^{\sigma_0,s}(\R^d)$ for some $\sigma_0 > 0$, $s\in \R$ and  $d=1, 2, 3$. Let $u$ be the global $C^\infty$ solution of \eqref{nls}. Then $u$ satisfies
\[
  u(t) \in G^{\sigma(t), s}(\R^d) \quad \text{for all } \ t\in \R,
\]
with the radius of analyticity $\sigma(t)$ satisfying an asymptotic lower bound
\[
\sigma(t) \ge \frac c{|t|} \quad \text {as} \ |t|\rightarrow \infty,
\]
where $c > 0$ is a constant depending on  $\|u_0\|_{G^{\sigma_0,s} (\R^d)}$, $\sigma_0$ and $s$.
\end{theorem}

By time reversal symmetry of \eqref{nls} we may from now on restrict ourselves to positive times $t\ge 0$. 
The first step in the proof of Theorem \ref{thm-gwp} is to show that in a short time interval $0 \le  t \le  \delta$, where $\delta > 0$ depends on the norm of the initial data, 
the radius of analyticity remains strictly positive. This is proved by a contraction argument involving energy estimates, 
Sobolev embedding and a multilinear estimate which will be given in the next section.
The next step is to improve the control on the growth of the solution in the time interval $[0, \delta]$, measured in the data norm $G^{\sigma_0, 1}$.
To achieve this we show that, although the conservation of  $G^{\sigma_0, 1}$-norm of solution does not hold exactly, it does hold in an approximate sense (see Section 3).
This approximate conservation law will allow us to iterate the local result and obtain Theorem \ref{thm-gwp}. 
This will be proved in Section 4.

\section{Preliminaries}

\subsection{ Function spaces}

Define the Bourgain space $X^{s,b}=X^{s,b}(\R^{1+d})$ by the norm
\begin{align*}
\| u \|_{X^{s,b}} &= \left\| \angles{ \xi }^{s} \angles{ \tau + |\xi|^2}^b \widetilde{u} (\xi, \tau) \right\|_{L^2_{\tau, \xi}},
\end{align*}
where $\widetilde u$ denotes the space-time Fourier transform given by
$$
\widetilde u(\tau, \xi)=\int_{\R^{1+d}} e^{ -i(t\tau+ x\xi)} u(t,x) \ dt dx.
$$
The restriction to time slab $(0, \delta)\times \mathbb R^d $ of the Bourgain space, denoted $X^{s, b}_\delta$, is a Banach space when equipped with the norm
$$
\| u \|_{X^{s, b}_\delta}  = \inf \left\{ \| v \|_{X^{s, b}}: \ v = u \text{ on }  (0, \delta) \times \mathbb{R}^d \right\}.
$$

In addition, we also need the Grevey-Bourgain space,  denoted $X^{\sigma, s,b}=X^{\sigma, s,b}(\R^{1+d})$, defined by the norm
\begin{align*}
\| u \|_{X^{\sigma, s, b}} &= \left\| e^{\sigma | D |}u \right\|_{X^{s, b}} .
\end{align*}
 In the case $\sigma=0$, this space coincides with the Bourgain space $ X^{s,b}$. 
The restrictions of $X^{\sigma, s, b}$ to a time slab $(0, \delta)\times \mathbb R^d$, denoted $X^{\sigma, s, b}_\delta$, is defined in a similar way as above.

\subsection{Linear estimates}

In this subsection we collect linear estimates needed to prove local existence of solution. The $X^{\sigma, s, b}$- estimates given below easily follows 
by substitution $u\rightarrow e^{\sigma | D |}u$ from the properties of $X^{s, b}$-spaces (and its restrictions).  In the case $\sigma=0$, the
proofs of the first two lemmas below can be found in section 2.6 of \cite{t06}, 
 whereas the third lemma follows by the argument used to prove Lemma 3.1 of \cite{ckstt04} and the fourth lemma is the standard energy estimate in 
 $X^{s, b}_\delta$-spaces.

\begin{lemma}\label{embed}
Let $\sigma \ge 0$, $s \in \mathbb{R}$ and $b > 1/2$. Then $X^{\sigma,s,b} \subset C(\mathbb R,G^{\sigma,s})$ and
\[
  \sup_{t \in \mathbb{R}} \| u(t) \|_{G^{\sigma, s}} \leq C \| u \|_{X^{\sigma, s, b}},
\]
where the constant $C > 0$ depends only on $b$.
\end{lemma}

\begin{lemma}\label{exponent}
Let $\sigma \ge 0$, $s \in \mathbb{R}$, $-1/2 < b < b' < 1/2$ and $\delta > 0$.  Then
\[
\| u \|_{X^{\sigma, s, b}_\delta} \leq C \delta^{b' - b} \| u \|_{X^{\sigma, s, b'}_\delta},
\]
where $C$ depends only on $b$ and $b'$.
\end{lemma}

\begin{lemma}\label{lm-restrict}
Let $\sigma \ge 0$, $s \in \mathbb R$, $-1/2 < b < 1/2$ and $\delta > 0$. Then for any time interval $I \subset [0,\delta]$ we have
\[
  \norm{\chi_{I} u}_{X^{\sigma,s,b}} \le C \norm{u}_{X^{\sigma,s,b}_\delta},
\]
where $\chi_I(t)$ is the characteristic function of $I$, and $C$ depends only on $b$.
\end{lemma}

Next, consider the linear Cauchy problem, for given $g(x,t)$ and $u_0(x)$,
\[\begin{cases}
iu_t + \Delta u = g, \\
u(0) = u_0.
\end{cases}\]
Let $W(t) = e^{it\Delta} $ be the solution group with Fourier symbol $e^{-it|\xi|^2}$.  Then
we can write the solution using the Duhamel formula
\[
u(t) = W(t)u_0 -i \int_0^t W(t - t') g(t') \ dt'.
\]
 Then $u$ satisfies the following $X^{\sigma, s, b}$ energy estimate.
\begin{lemma}\label{lm-linear}
Let $\sigma \ge 0$, $s \in \mathbb{R}$, $1/2 < b \leq 1$ and $0 < \delta \leq 1$.  Then for all $u_0 \in G^{\sigma, s}$ and $F \in X^{\sigma, s, b-1}_\delta$, we have the estimates
\[\begin{aligned}
& \qquad \qquad \| W(t) u_0 \|_{X^{\sigma, s, b}_\delta} \leq C \| u_0 \|_{G^{\sigma, s}},
\\
& \left\| \int_0^t W(t - t') g(t')\ dt' \right\|_{X^{\sigma, s, b}_\delta} \leq C \| g \|_{X^{\sigma, s, b - 1}_\delta},
\end{aligned}\]
where the constant $C > 0$ depends only on $b$.
\end{lemma}

\begin{definition}
A pair $(q, r)$ of exponents are called admissible if $2\le q , r \le \infty$, 
$$
\frac2q + \frac dr=\frac d2 \quad \text{and} \quad (q,r,d)\neq (2, \infty, 2).
$$
\begin{lemma}[see \cite{t06}]\label{lm-srtz} 
Let $d\ge 2$ and $(q, r)$ be an admissible pair. Then we have the Strichartz estimate
\begin{equation}
\label{strz-est}
 \| W(t) u_0 \|_{L_t^qL_x^r(\R^{1+d})} \leq C \| u_0 \|_{L^2(\R^{d})}.
\end{equation}
Moreover, for any $b>\frac12$ and $u\in X^{0, b}(\R^{1+d})$ we have
\begin{equation}
\label{strz-transfer}
 \| u \|_{L_t^qL_x^r(\R^{1+d})} \leq C \| u \|_{X^{0, b} (\R^{1+d})}.
\end{equation}
\end{lemma}

\end{definition}

\subsection{Multilinear estimates and local result}

By duality, H\"{o}lder, Sobolev and the Strichartz estimate \eqref{strz-transfer} we obtain 
the following multilinear estimates. The proof will be given in the last section.

\begin{lemma}\label{lm-nonlinearest}
 Let $U_j$ denotes $u_j$ or $\overline u_j$. Let $d=1, 2, 3$,  $\sigma\ge 0$ and $b>\frac12$. Then we have the estimates
\begin{align}
\label{nonlinearest-1}
\|   \prod_{j=1}^3 U_j   \|_{X^{0, -b} }&\ls  \|   u_1  \|_{X^{1, b} }  \|   u_2  \|_{X^{0, b}}   \|   u_3  \|_{X^{0, b}} ,
\\
\label{nonlinearest-2}
\|   \prod_{j=1}^3U_j  \|_{L^2_{t,x}  }&\ls    \|   u_1  \|_{X^{1, b} } \|   u_2  \|_{X^{1, b} }  \|   u_3  \|_{X^{0, b} },
   \\
\label{nonlinearest-3}
\|    \prod_{j=1}^3U_j \|_{X^{\sigma, 1, 0}  }&\ls \|   u_1  \|_{X^{\sigma, 1, b} } \|   u_2  \|_{X^{\sigma, 1, b} }  \|   u_3  \|_{X^{\sigma, 1, b} }.
\end{align}
\end{lemma}

By Picard iteration in the $X^{\sigma, 1, b}_\delta$-space and application of Lemma \ref{lm-linear}, Lemma \ref{lm-restrict} and \eqref{nonlinearest-3} to the iterates one obtains the following local result (for details see \cite[proof of Theorem 1 therein]{{sd16}}).
\begin{theorem}\label{thm-lwp}
Let $d=1, 2, 3$, $\sigma > 0$ and  \footnote{ We use the notation $a\pm=a\pm \varepsilon$ for sufficiently small $\varepsilon>0$. }  $b=\frac12+$. Then for any $u_0 \in G^{\sigma,1}$ there exists a time $\delta > 0$ and a unique solution $u$ of \eqref{nls} on the time interval $(0,\delta)$ such that
\[
u \in C([0,\delta], G^{\sigma,1} ).
\]
Moreover, the solution depends continuously on the data $u_0$, and we have
\[
  \delta = c_0(1+\|u_0\|_{G^{\sigma, 1}  })^{-4-}
\]
for some constant $c_0 > 0$. Furthermore, the solution $u$ satisfies the bound 
\begin{equation} \label{solnbound}
\| u\|_{X^{ \sigma, 1, b}_\delta}\le C \| u_0\|_{G^{\sigma, 1}  },  
\end{equation}  
where $C$ depends only on $b$.
\end{theorem}

\section{Almost conservation law }
Define
$$
A_\sigma(t)= \| u(t) \|^2_{G^{\sigma, 1}}+ \frac12\| e^{\sigma |D|  } u(t)\|_{L^4}^4.
$$
For $\sigma=0$ we have from \eqref{consv-mass} and \eqref{consv-energy} the conservation
$$
A_0(t)=A_0(0) \quad \text{for all} \ t.
$$
However, this fails to hold for $\sigma>0$. In what follows we will nevertheless prove, for $\delta$ as in Theorem \ref{thm-lwp}, the approximate conservation
$$
\sup_{t\in [0,  \delta]} A_\sigma(t)=A_\sigma(0)+ R_{\sigma}(0),
$$
where the quantity $ R_{\sigma}(0)$ satisfies the bound $$ R_{\sigma}(0)\le C \sigma A_\sigma^2(0) \left(1+  A_\sigma(0)\right).$$
 In the limit as $\sigma\rightarrow 0 $, we have 
 $R_\sigma(0) \rightarrow 0$, and hence we recover the conservation $A_0(t)=A_0(0)$.

 To this end, we note from \eqref{solnbound} that
 \begin{equation} \label{solnbound-A}
\| u\|_{X^{ \sigma, 1, b}_\delta}\le C [A_{\sigma}(0)]^\frac12.
\end{equation}  
 
\begin{theorem}\label{thm-approx}
Let $d=1, 2, 3$ and $\delta$ be as in Theorem \ref{thm-lwp}. There exists $C > 0$ such that for any $\sigma > 0$ and any solution
 $u \in X^{\sigma, 1, b}_\delta$ to the Cauchy problem \eqref{nls} on the time interval  $[0,\delta]$, we have the estimate
\begin{equation}\label{approx1}
\sup_{t\in [0,  \delta]} A_\sigma(t) \leq A(0) + C \sigma \| u \|^4_{X^{\sigma, 1, b}_\delta} \left(1+  \| u \|^2_{X^{\sigma, 1, b}_\delta}\right).
\end{equation} 
Moreover, we have
\begin{equation}\label{approx2}
 \sup_{t\in [0,  \delta]} A_\sigma (t) \leq A_\sigma(0) + C \sigma A_\sigma^2(0) \left(1+  A_\sigma(0)\right).
\end{equation} 
\end{theorem}

\begin{proof}
It suffices to prove \eqref{approx1}
since the estimate \eqref{approx2} follows from \eqref{approx1} and \eqref{solnbound-A}.  We prove \eqref{approx1} in two steps.

\subsection*{Step 1}
Let $v(t,x)=e^{\sigma |D| } u (t,x)$.  
Applying $e^{\sigma |D| }$ to \eqref{nls} we obtain
\begin{equation}\label{nlsm}
iv_t + \Delta v  = |v|^2 v + f(v),
\end{equation}
where
$$
f(v)=-\left\{  |  v|^2v  -e^{\sigma |D| } \left(    |e^{-\sigma |D| } v|^2 e^{-\sigma |D| } v\right)  \right\}.
$$
Using \eqref{nlsm} we have
$$
\re(\overline v v_t )+ \im( \overline v \Delta v )  = \im(\overline v  f(v))
$$
or equivalently
$$
(|v|^2 )_t+ 2\im( \nabla \cdot (\overline v \nabla v ))  = 2\im(\overline v  f(v)),
$$
where we used the fact $\overline v \Delta v=\nabla \cdot (\overline v \nabla v )- | \nabla v |^2$.
We may assume \footnote{
In general, this property holds by approximation using the monotone convergence theorem and the Riemann-Lebesgue Lemma  whenever $u\in X_\delta^{ \sigma, 1, b}$
(see the argument in \cite[pp. 9 ]{sd16}). }  $v, \nabla v$ and $\Delta v$ decays to zero as $|x|\rightarrow \infty$.
Integrating in space we obtain
$$
\frac{d}{dt}\int_{\R^d}  |v|^2 dx= 2\im \int_{\R^d}  \overline v f(v) dx.
$$
Now integrating in time over the interval $[0, \delta]$, we obtain
\begin{align*}
\int_{\R^d}  |v(\delta)|^2 dx= \int_{\R^d}   |v(0)|^2 dx+2\im \int_{\R^{1+d}}  \chi_{[0, \delta]}(t) \overline vf(v)  dx dt.
\end{align*}
Hence
$$
\| u(\delta)\|_{G^{\sigma,0}}^2=\| u(0)\|_{G^{\sigma, 0}}^2 +2\im \int_{\R^{1+d}}  \chi_{[0, \delta]}(t) \overline vf(v)  dx dt.
$$
We now use H\"{o}lder, Lemma \ref{lm-restrict} and Lemma \ref{lm-f-est} below to estimate the integral on the right hand side as
\begin{align*}
\left |\int_{\R^{1+d}} \chi_{[0, \delta]}(t) \overline vf(v)  dx dt \right|
&\le\|  v\|_{L_{t, x}^2 [0,\delta] }  \|  \overline {f(v)}\|_{ L_{t, x}^2 [0,\delta] }
\\
&\le \| v\|_{X^{0, b}_\delta}  \cdot  C \sigma  \| v\|_{X^{ 1, b}_\delta}^3
\\
& \le C\sigma  \| u\|_{X^{ \sigma, 1, b}_\delta}^4.
\end{align*}

Thus
\begin{equation}\label{approxmassCosv}
\| u(\delta)\|_{G^{\sigma, 0}}^2\le \| u(0)\|_{G^{\sigma,0}}^2+C\sigma \| u\|_{X^{ \sigma, 1, b}_\delta}^4.
\end{equation}

\subsection*{Step 2}
Differentiating \eqref{nlsm} we have
\begin{equation}\label{nlsm1}
i\nabla v_{t} + \Delta \nabla v   = \nabla  ( |v|^2 v )+  \nabla f(v)
\end{equation}
from which we obtain
$$
\re(\overline {\nabla v } \cdot \nabla v_t )+ \im\left( \overline {\nabla v } \cdot \Delta\nabla  v \right)  
= \im\left( \overline {\nabla v }\cdot  \nabla  ( |v|^2 v )+ \overline {\nabla v } \cdot  \nabla f(v) \right).
$$
We have
$$\overline {\nabla v } \cdot \Delta\nabla  v = \nabla\cdot ( \overline {\nabla v} \Delta v)  -| \Delta v|^2  $$
and using \eqref{nlsm} we write
\begin{align*}
 \overline {\nabla v} \cdot \nabla ( |v|^2 v ) &= \nabla \cdot( \overline {\nabla v}  |v|^2 v )- \overline {\Delta v}  |v|^2 v
\\
&=\nabla \cdot( \overline {\nabla v}  |v|^2 v )-i |v|^2 v \overline v_t - |v|^6 -  |v|^2 v \overline{ f(v)}.
\end{align*}
 It then follows
\begin{align*}
&\frac{d}{dt} \left( |\nabla v|^2 +\frac12  |v|^4\right)+  2\im\left(  \nabla\cdot ( \overline {\nabla v} \Delta v)  \right) 
\\
 & \qquad =
 2\im\left(\nabla \cdot( \overline {\nabla v}  |v|^2 v )  - |v|^2 v \overline{ f(v)} +    \overline {\nabla v } \cdot  \nabla f(v) \right).
\end{align*} 
Integrating this equation in space we get
$$
 \frac{d}{dt} \int_{\R^d} |\nabla v |^2 + \frac12 |v|^4 dx= -2\im \int_{\R^d}  (  |v|^2 v \overline{ f(v)}- \overline {\nabla v} \cdot \nabla  f(v) )  dx.
$$
Integrating in time over the interval $[0, \delta]$ we obtain
\begin{align*}
\int_{\R^d} |\nabla v (\delta)|^2   + \frac12 |v(\delta)|^4  dx&= \int_{\R^d}  |\nabla v(0)|^2  + \frac12 |v(0)|^4 dx
\\
\qquad &-2\im  \int_{\R^{1+d}} \chi_{[0, \delta]}(t)  (  |v|^2 v \overline{ f(v)}+ \overline {\nabla v} \cdot \nabla  f(v))  dx dt.
\end{align*}

Now by H\"{o}lder, \eqref{nonlinearest-2}, Lemma \ref{lm-restrict} and Lemma \ref{lm-f-est} below we estimate 
\begin{align*}
\left |\int_{\R^{1+d}} \chi_{[0, \delta]}(t) |v|^2 v \overline {f(v)} dx dt \right|
&\le\|  |v|^2 v\|_{ L_{t, x}^2 [0,\delta] }  \| \overline {f(v)}\|_{ L_{t, x}^2 [0,\delta] }
\\
&\le \| v\|^3_{X^{1, b}_\delta}  \cdot  C \sigma \| v\|_{X^{ 1, b}_\delta}^3
\\
& \le C\sigma  \| u\|_{X^{ \sigma, 1, b}_\delta}^6
\end{align*}
and 
\begin{align*}
\left |\int_{\R^{1+d}} \chi_{[0, \delta]}(t) \overline {\nabla v} \cdot \nabla  f(v)  dx dt \right|
&\le\|   \nabla v\|_{X^{0, b}_\delta }  \|    \overline {\nabla  f(v)}\|_{ X^{0, -b}_\delta }
\\
&\le \| v\|_{X^{1, b}_\delta}  \cdot  C \sigma \| v\|_{X^{ 1, b}_\delta}^3
\\
& \le C\sigma  \| u\|_{X^{ \sigma, 1, b}_\delta}^4.
\end{align*}

Thus
\begin{equation}\label{approxengCosv}
\begin{split}
\| \nabla u(\delta) \|^2_{G^{\sigma, 0}}+ \frac12\| e^{\sigma |D|  } u(\delta)\|_{L^4}^4 &\le 
\| \nabla u(0) \|^2_{G^{\sigma, 0}}+
 \frac12\| e^{\sigma |D|  } u(0)\|_{L^4}^4 
 \\
 &\qquad +C \sigma \| u\|_{X^{ \sigma, 1, b}_\delta}^4 \left(1+\| u\|_{X^{ \sigma, 1, b}_\delta}^2\right).
\end{split}
\end{equation}

We conclude from \eqref{approxmassCosv} and \eqref{approxengCosv} that
\begin{equation}\label{approxCosv-me}
A_\sigma (\delta)\le A_\sigma(0) + C\sigma \| u\|_{X^{ \sigma, 1, b}_\delta}^4 \left ( 1 +\| u\|_{X^{ \sigma, 1, b}_\delta}^2\right).
\end{equation}

\end{proof}

\begin{lemma}\label{lm-f-est}
Let
$$
f(v)=-\left\{  |  v|^2v  -e^{\sigma |D| } \left(    |e^{-\sigma |D| } v|^2 e^{-\sigma |D| } v\right)  \right\}.
$$
For all
$b>\frac12$ we have
\begin{align}
 \label{fest1}
  \| \overline { f(v)}\|_{L_{t, x}^2 }&\le C\sigma  \| v\|_{X^{1, b}}^3 ,
  \\
\label{fest1}
 \| \overline {\nabla f(v) }\|_{X^{0, -b}}&\le C \sigma  \| v\|_{X^{1, b}}^3 ,
\end{align}
for some $C>0$ which is independent of $\sigma$.
\end{lemma}

\begin{proof}
Taking the space-time Fourier Transform of $\overline f$ we have
\begin{align*}
|\widetilde{\overline {f(v) }}(\tau, \xi) |&=\left |\int_{\ast}
 \left\{    1- e^{-\sigma( \sum_{j=1}^{3} |\xi_j| - |\xi| )}   \right\} \widetilde{v }(\tau_1, \xi_1)\overline{\widetilde{ v }(\tau_2, \xi_2) \widetilde{ v }(\tau_3, \xi_3)}
 d\tau_1d\xi_1 d\tau_2d\xi_2\right|
\end{align*}
where we used $\ast$ to denote the conditions $\xi= \xi_1-\xi_2-\xi_3 $ and $\tau= \tau_1-\tau_2-\tau_3 $. 
Now denote the minimum, medium and maximum 
 of $\{ |\xi_1|,  |\xi_2|, |\xi_3|\}$ by
 $\xi_{\text{min}}$, $\xi_{\text{med}}$ and $\xi_{\text{max}}$.
Then we have
\begin{align*}
 1- e^{-\sigma( \sum_{j=1}^{3} |\xi_j| - |\xi| )} &\le \sigma \left( \sum_{j=1}^{3} |\xi_j| - |\xi| \right)
\\
&= \sigma \frac{  (\sum_{j=1}^{3} |\xi_j| )^2  - |\xi | ^2}{  \sum_{j=1}^{3} |\xi_j|   +| \xi |}
\\
  &\le 12 \sigma \frac{ \xi_{\text{med}} \cdot \xi_{\text{max}} }{  \xi_{\text{max}}}=12 \sigma \xi_{\text{med}} .
\end{align*}
Consequently, 
\begin{align*}
|\widetilde{\overline {f(v)} }(\tau, \xi) |
 &\le 12 \sigma\int_{\ast}   \xi_{\text{med}} 
 |\widetilde{v  }(\tau_1, \xi_1)||\widetilde{v }(\tau_2, \xi_2)||\widetilde{ v }(\tau_3, \xi_3)|
 d\tau_1d\xi_1 d\tau_2d\xi_2.
\end{align*}
Let
$$
\widetilde{w_j }(\tau_j, \xi_j)=|\widetilde{v  }(\tau_j, \xi_j)|.
$$
By symmetry, we may assume $ |\xi_1|\le  |\xi_2| \le |\xi_3|$, and hence $\xi_{\text{med}}=   |\xi_2|$. So we use \eqref{nonlinearest-2} to obtain
\begin{align*}
 \| \overline { f(v)}\|_{L^2_{t,x}}&= \left\|   \widetilde{\overline { f(v)} }(\tau, \xi) \right\|_{L^2_{\tau, \xi}}
\\
&\ls  \sigma  \left\|  \int_{\ast}  \angles{\xi_3}  \widetilde{w_1  }(\tau_1, \xi_1) \overline{\widetilde{ w_2 }(\tau_2, \xi_2) \widetilde{w_3 }(\tau_3, \xi_3)}
 d\tau_1d\xi_1 d\tau_2d\xi_2\right \|_{L^2_{\tau, \xi}}
  \\
  &\ls \sigma \left\|    w_1 \cdot \overline{ w_2} \cdot \overline{  \angles{D} w_3 }  \right\|_{L^2_{t,x}}
  \\
  &\ls \sigma   \|     w_1 \|_{ X^{1, b} }    \|    w_2\|_{ X^{1, b} }  \|  \angles{D} w_3 \|_{  X^{0, b}} 
  \\
  &\ls \sigma  \|     v \|_{X^{1, b}}^3.
\end{align*}
Similarly, we use \eqref{nonlinearest-1} to obtain
\begin{align*}
 \| \overline {\nabla f(v)}\|_{X^{0, -b}}&= \left\| \|\xi| \angles{\tau+|\xi|^2}^{-b}  \widetilde{\overline {f(v)} }(\tau, \xi) \right\|_{L^2_{\tau, \xi}}
\\
&\ls  \sigma  \left\| \angles{\tau+|\xi|^2}^{-b} \int_{\ast}  \angles{\xi_2} \angles{\xi_3} \widetilde{w_1  }(\tau_1, \xi_1)
 \overline{\widetilde{ w_2 }(\tau_2, \xi_2) \widetilde{w_3 }(\tau_3, \xi_3)}
 d\tau_1d\xi_1 d\tau_2d\xi_2\right \|_{L^2_{\tau, \xi}}
  \\
  &\ls \sigma \left\|    w_1 \cdot \overline{\angles{D} w_2} \cdot \overline{ \angles{D} w_3 }  \right\|_{X^{0, -b}}
  \\
  &\ls \sigma   \|     w_1 \|_{ X^{1, b} }    \|  \angles{D}   w_2\|_{ X^{0, b} }  \|  \angles{D} w_3 \|_{ X^{0, b} } 
  \\
  &\ls \sigma  \|     v \|_{X^{1, b}}^3.
\end{align*}

\end{proof}

\section{Proof of Theorem \ref{thm-gwp} }

We closely follow the argument in \cite{sd16}. 
First we consider the case $s=1$.
The general case, $s\in \R$, will essentially reduce to $s=1$ as shown in the next subsection. 

\subsection{Case $s=1$}

Let $u_0=u(0) \in G^{\sigma_0, 1 }(\R^d)$ for some $\sigma_0 > 0$, where $d=1,2,3$. Then
by Gagliardo-Nirenberg inequality we have
\begin{align*}
A_{\sigma_0}(0)&=\|  u_0\|^2_{G^{\sigma_0, 1}}+ \frac12 \| e^{\sigma_0 |D|  } u_0\|_{L^4}^4
\\
&\le \|  u_0\|^2_{G^{\sigma_0, 1}}+c \| \nabla(e^{\sigma_0 |D|  }  u_0)\|^d_{L^2}  \|e^{\sigma_0 |D|  }u_0\|^{4-d}_{L^2}
 \\ 
 &\le \|  u_0\|^2_{G^{\sigma_0, 1}}+ c( \| \nabla u_0\|^{2d}_{G^{\sigma_0, 0}}+  \|u_0\|^{8-2d}_{G^{\sigma_0,0}})
 \\
 &<\infty.
\end{align*}

To construct a solution on $[0, T]$ for  arbitrarily large $T$, we will apply
 the approximate conservation law in Theorem \ref{thm-approx} so as to repeat the local result 
on successive short time intervals to reach $T$, by adjusting the strip width parameter $\sigma$ according to the size of $T$.
By employing this strategy we will show that the solution $u$ to \eqref{nls} satisfies
\begin{equation}
\label{uT}
  u(t) \in G^{\sigma(t), 1 } \quad \text{for all }  \  t\in [0,  T] ,
 \end{equation}
with
\begin{equation}
\label{siglb}
\sigma(t) \ge \frac cT ,
\end{equation}
where $c > 0$ is a constant depending on $\|u_0\|_{G^{\sigma_0, 1}}$ and $\sigma_0$.

By Theorem \ref{thm-lwp} there is 
a solution $u$ to \eqref{nls} satisfying
\[
  u(t) \in G^{\sigma_0,1} \quad \text{for all }  \  t\in [0,  \delta] 
\]
where
\begin{equation}\label{delta}
  \delta = c_0 (1+A_{\sigma_0}(0))^{-4-} .
\end{equation}

Now fix $T$ arbitrarily large. We shall apply the above local result and Theorem \ref{thm-approx} repeatedly, with a uniform time step $\delta$ as in \eqref{delta},  and prove 
\begin{equation}\label{keybound}
\sup_{t\in [0, T]}  A_\sigma  (t) \le 2A_{\sigma_0} (0)
\end{equation}
for $\sigma$ satisfying \eqref{siglb}. 
Hence we have $A_\sigma(t) < \infty$ for $t \in [0,T]$, which in turn implies $u(t) \in G^{\sigma(t), 1 }$, and this completes the proof of \eqref{uT}--\eqref{siglb}.

It remains to prove \eqref{keybound} which shall do as follows. Choose $n \in \mathbb N$ so that $T \in [n\delta,(n+1)\delta)$. Using induction we can show for 
any $k \in \{1,\dots,n+1\}$ that
\begin{align}
  \label{induction1}
  \sup_{t \in [0, k\delta]} A_\sigma(t) &\le A_\sigma(0) + 8 C\sigma k A^2_{\sigma_0}(0)\{  1+ A_{\sigma_0}(0)\},
  \\
  \label{induction2}
  \sup_{t \in [0,k\delta]} A_\sigma(t) &\le 2A_{\sigma_0}(0),
\end{align}
provided $\sigma$ satisfies 
\begin{equation}\label{sigma}
  \frac{16 T}{\delta} C\sigma A_{\sigma_0}(0)\{  1+ A_{\sigma_0}(0)\} \le 1.
\end{equation}

Indeed, for $k=1$, we have from Theorem \ref{thm-approx} that
\begin{align*}
  \sup_{t \in [0, \delta]} A_\sigma(t) &\le A_\sigma(0) +  C\sigma A^2_{\sigma}(0)\{  1+ A_{\sigma}(0)\}
  \\
  &\le A_\sigma(0) +  
  C\sigma A^2_{\sigma_0}(0)\{  1+ A_{\sigma_0}(0)\}
\end{align*}
where we used $A_\sigma(0) \le A_{\sigma_0}(0) $. This in turn implies \eqref{induction2} provided 
$$C\sigma A_{\sigma_0}(0)\{  1+ A_{\sigma_0}(0)\}\le 1$$
 which holds by \eqref{sigma} since 
$T>\delta$.

Now 
assume \eqref{induction1} and \eqref{induction2} hold for some $k \in \{1,\dots,n\}$. Then applying Theorem \ref{thm-approx}, \eqref{induction2} and \eqref{induction1}, respectively, we obtain
\begin{align*}
  \sup_{t \in [k\delta, (k+1)\delta]} A_\sigma(t) &\le A_\sigma(k\delta) +  C\sigma  A^2_{\sigma}(k\delta) \{ 1+ A_{\sigma}(k\delta)\}
  \\
   &\le A_\sigma(k\delta) +  8C\sigma A^2_{\sigma_0}(0)\{  1+ A_{\sigma_0}(0)\}
   \\
      &\le   A_\sigma(0) +  8C\sigma (k+1) A^2_{\sigma_0}(0)\{  1+ A_{\sigma_0}(0)\}.
\end{align*}
Combining this with the induction hypothesis
 \eqref{induction1} (for $k$) we obtain
 \begin{align*}
  \sup_{t \in [0, (k+1)\delta]} A_\sigma(t) 
      &\le   A_\sigma(0) +8C\sigma (k+1) A^2_{\sigma_0}(0)\{  1+ A_{\sigma_0}(0)\}
\end{align*}
 which proves \eqref{induction1} for $k+1$. This also implies \eqref{induction2} for $k+1$ provided
 \begin{align*}
  8 ( k+1) C\sigma A_{\sigma_0}(0) \{  1+ A_{\sigma_0}(0)\}\le 1.
\end{align*}
 But the latter follows from \eqref{sigma} since 
 $$
 k+1\le n+1\le \frac T\delta+ 1 \le \frac {2T}\delta.
 $$

Finally,  the condition \eqref{sigma} is satisfied for $\sigma$ such that
    $$
    \frac{16 T}{\delta} C\sigma A_{\sigma_0}(0)\{  1+ A_{\sigma_0}(0)\} =1.
    $$
Thus, 
$$
\sigma=\frac{c_1}T , \ \ \text{where} \ c_1=  \frac{c_0}{16C  A_{\sigma_0}(0) (1+A_{\sigma_0}(0))^{5+}}
$$
which gives \eqref{siglb} if we choose $c\le c_1$.

\subsection{The general case: $s\in \R$}

For any $s\in \R$ we use the embedding \eqref{Gembedding} to get
\[
  u_0 \in G^{\sigma_0,s} \subset G^{\sigma_0/2, 1}.
\]
From the local theory there is a $\delta=\delta \left(   A_{\sigma_0/2}(0) \right)$ such that
\[
  u(t)\in  G^{\sigma_0/2, 1} \quad \text{for } \ 0\le t \le \delta.
\]
Fix an arbitrarily large $ T$. From the case $s=1$ in the previous subsection we have
\[
  u(t) \in G^{ 2 a_0 T^{-1} , 1}  \quad \text{for $ t\in [0,   T]$},
\]
where $a_0 > 0$ depends on $A_{\sigma_0/2}(0)$ and $\sigma_0$. Applying again the embedding \eqref{Gembedding} we conclude that
\[
  u(t) \in G^{  a_0 T ^{-1}, s }  \quad \text{for $ t\in [0,   T]$},
\]
 completing the proof of Theorem \ref{thm-gwp}.

\section{Proof of Lemma \ref{lm-nonlinearest}}

\subsection{Estimate \eqref{nonlinearest-1} }
In the case of $d=1$ we have the stronger estimate (see \cite[Colloraly 3.1]{g01})
$$
\|   U_1  U_2 U_3 \|_{L^2_{t,x}(\R^{1+1})  }\ls   \prod_{j=1}^3\|   u_j  \|_{X^{0, b} (\R^{1+1})},
$$ 
which also implies \eqref{nonlinearest-2}.

Now assume $d=2, 3$. By duality \eqref{nonlinearest-1} reduces to
$$
\left| \int_{\R^{1+d} }  U_1 U_2 U_3 \overline u_4  dt dx\right| \ls \|   u_1  \|_{X^{1, b} (\R^{1+d})} \prod_{j=2}^4 \|   u_j  \|_{X^{0, b} (\R^{1+d})}.
$$
 By H\"{o}lder and \eqref{strz-transfer} we have
\begin{align*}
\left| \int_{\R^{1+2} }    U_1 U_2 U_3  \overline u_4  dt dx\right| &\le \prod_{j=1}^4 \|   u_j  \|_{L^4_{t,x} (\R^{1+2})}
\\
&\le  \prod_{j=1}^4 \|   u_j  \|_{X^{0, b} (\R^{1+2})}.
\end{align*}
Similarly, by H\"{o}lder, Sobolev and \eqref{strz-transfer} we have
\begin{align*}
\left| \int_{\R^{1+3} }      U_1   U_2 U_3  \overline u_4  dt dx\right| &\le 
\|    u_1  \|_{L^\infty_tL^6_x(\R^{1+3})  } \prod_{j=2}^3\|   u_j  \|_{L^2_tL^6_x (\R^{1+3}) } \|   u_4  \|_{L^\infty_tL^2_x (\R^{1+3}) } 
\\
&\le \|   \angles{D}  u_1  \|_{L^\infty_tL^2_x (\R^{1+3}) } \prod_{j=2}^3\|   u_j  \|_{L^2_tL^6_x (\R^{1+3}) } \|   u_4  \|_{L^\infty_tL^2_x (\R^{1+3}) } 
\\
&\ls  \|   u_1  \|_{X^{1, b} (\R^{1+3})}  \prod_{j=2}^4 \|   u_j  \|_{X^{0, b} (\R^{1+3})}  .
\end{align*}

\subsection{Estimate \eqref{nonlinearest-2} }
Assume $d=2,3$. 
By duality \eqref{nonlinearest-2} reduces to
$$
\left| \int_{\R^{1+d} }  U_1 U_2 U_3 \overline u_4  dt dx\right| \ls  \prod_{j=1}^2\|   u_j  \|_{X^{1, b} (\R^{1+d})}  \|   u_3  \|_{X^{0, b} (\R^{1+d})}
 \|   u_4  \|_{L^2_{t,x} (\R^{1+d})}.
$$

By H\"{o}lder, Sobolev and \eqref{strz-transfer} we obtain
\begin{align*}
\left| \int_{\R^{1+2} }     U_1  U_2 U_3  \overline u_4  dt dx\right| &\le 
\prod_{j=1}^2\|    u_j  \|_{L^8_{t,x}(\R^{1+2})  } \|   u_3  \|_{L^4_{t,x}(\R^{1+2}) } \|   u_4  \|_{L^2_{t,x} (\R^{1+2}) } 
\\
&\le \prod_{j=1}^2\|   \angles{D}^{\frac12} u_j  \|_{L^8_t L^{\frac83}_x (\R^{1+2})  } \|   u_3  \|_{L^4_{t,x}(\R^{1+2}) } \|   u_4  \|_{L^2_{t,x} (\R^{1+2}) } 
\\
&\ls \prod_{j=1}^2 \|   u_j  \|_{X^{\frac12, b} (\R^{1+2})}   \|   u_3  \|_{X^{0, b} (\R^{1+2})}
 \|   u_4  \|_{L^2_{t,x} (\R^{1+2})}.
\end{align*}

Similarly,
\begin{align*}
\left| \int_{\R^{1+3} }     U_1   U_2 U_3  \overline u_4  dt dx\right| &\le 
\prod_{j=1}^2\|    u_j  \|_{L^\infty_t L^6_x (\R^{1+3})  } \|   u_3  \|_{L^2_t L^6_x(\R^{1+3}) } \|   u_4  \|_{L^2_{t,x} (\R^{1+3}) } 
\\
&\le \prod_{j=1}^2\|   \angles{D} u_j  \|_{L^\infty_t L^2_x (\R^{1+3})  } \|   u_3  \|_{L^2_t L^6_x(\R^{1+3}) } \|   u_4  \|_{L^2_{t,x} (\R^{1+3}) } 
\\
&\ls \prod_{j=1}^2 \|   u_j  \|_{X^{1, b} (\R^{1+3})}  \|   u_3  \|_{X^{0, b} (\R^{1+3})}
 \|   u_4  \|_{L^2_{t,x} (\R^{1+3})}.
\end{align*}

\subsection{Estimate \eqref{nonlinearest-3} }
First assume $\sigma=0$.
By Leibniz rule and symmetry it suffices to show
$$
\|   U_1 U_2 \angles{D}U_3 \|_{L^2_{t,x}  }\ls  \prod_{j=1}^3 \|   u_j  \|_{X^{1, b} } .
$$
But this follows from \eqref{nonlinearest-2}. Next assume  $\sigma>0$.
W.l.o.g assume $U_1=u_1$,  $ U_2=\overline u_2 $ and $ U_3= \overline u_3 $. Let
$$
v_j=e^{\sigma |D|} u_j.
$$
Then \eqref{nonlinearest-3} reduces to 
\begin{equation}\label{nonlinearest-31}
\|   e^{\sigma |D|} \left( e^{-\sigma |D|}v_1  \cdot  \overline{ e^{-\sigma |D|}v_2 \cdot e^{-\sigma |D|}v_3  } \right) \|_{ X^{1,0} }\ls  \prod_{j=1}^3 \|   v_j  \|_{X^{1, b} } .
\end{equation}

 Let 
$$
\widetilde{w_j }(\tau_j, \xi_j)=|\widetilde{v  }(\tau_j, \xi_j)|.
$$
Then taking the space-time Fourier Transform and using the fact that $e^{\sigma ( |\xi| -\sum_{j=1}^4 |\xi_j| )}\le 1$, which follows from the triangle inequality, we obtain 
\begin{align*}
\text{L.H.S} \ \eqref{nonlinearest-31}&\le \left\|  \int_{\ast}  \angles{\xi}  |\widetilde{v_1  }(\tau_1, \xi_1)|| \widetilde{ v_2 }(\tau_2, \xi_2) | |\widetilde{v_3 }(\tau_3, \xi_3)|
 d\tau_1d\xi_1 d\tau_2d\xi_2\right \|_{L^2_{\tau, \xi}}
 \\
 &\le \left\|  \int_{\ast}  \angles{\xi} \widetilde{w_1  }(\tau_1, \xi_1)\overline{\widetilde{ w_2 }(\tau_2, \xi_2) \widetilde{w_3 }(\tau_3, \xi_3)}
 d\tau_1d\xi_1 d\tau_2d\xi_2\right \|_{L^2_{\tau, \xi}}
 \\
  &\ls  \left\|    w_1 \cdot \overline{ w_2} \cdot \overline{  w_3 }  \right\|_{X^{1,0}}
  \\
  &\ls    \prod_{j=1}^3 \|   w_j  \|_{X^{1, b} }
 = \prod_{j=1}^3 \|   v_j  \|_{X^{1, b} },
\end{align*}
where in the last line we used the estimate for $\sigma=0$.

\end{document}